\documentclass[12pt]{amsart}
\usepackage{graphics}
\usepackage{amsfonts}
\usepackage{amsmath, amssymb}
\usepackage{tikz-cd}
\usepackage[T1]{fontenc}
\usepackage[utf8]{inputenc}
\usepackage{mathtools}
\usepackage{comment}

\usepackage{geometry}
\newcommand{\Z}{{\mathbb Z}}

\newcommand{\C}{{\mathbb C}}

\newtheorem{thm}{Theorem}[section]
\newtheorem{lemma}[thm]{Lemma}
\newtheorem{cor}[thm]{Corollary}
\newtheorem{prop}[thm]{Proposition}
\newtheorem{defn}[thm]{Definition}

\newtheorem{remark}[thm]{Remark}

\title[Denniston partial difference sets exist in the odd prime case]{Denniston partial difference sets exist in the odd prime case}

\author{James A. Davis}
\address{Department of Math \& Computer Science, 212 Jepson Hall, 221 Richmond Way, University of Richmond, VA 23173, USA}
\email{jdavis@richmond.edu}

\author{Sophie Huczynska}
\address{School of Mathematics and Statistics, Mathematical Institute, University of St Andrews, North Haugh, St Andrews, Fife, KY16 9SS, Scotland, UK}
\email{sh70@st-andrews.ac.uk}

\author{Laura Johnson}
\address{School of Mathematics and Statistics, Mathematical Institute, University of St Andrews, North Haugh, St Andrews, Fife, KY16 9SS, Scotland, UK}
\email{lj68@st-andrews.ac.uk}

\author{John Polhill}
\address{Department of Mathematics, Computer Science and Digital Forensics, Commonwealth University, Bloomsburg, PA, USA}
\email{jpolhill@commonwealthu.edu}


\begin{document}
\maketitle

\begin{abstract}
Denniston constructed partial difference sets (PDSs) with the parameters $(2^{3m}, (2^{m+r} - 2^m + 2^r)(2^m-1), 2^m-2^r+(2^{m+r}-2^m+2^r)(2^r-2), (2^{m+r}-2^m+2^r)(2^r-1))$ in elementary abelian groups of order $2^{3m}$ for all $m \geq 2, 1 \leq r < m$.  These correspond to maximal arcs in Desarguesian projective planes of even order. In this paper, we show that - although maximal arcs do not exist in Desarguesian projective planes of odd order - PDSs with the Denniston parameters $(p^{3m}, (p^{m+r} - p^m + p^r)(p^m-1), p^m-p^r+(p^{m+r}-p^m+p^r)(p^r-2), (p^{m+r}-p^m+p^r)(p^r-1))$ exist in all elementary abelian groups of order $p^{3m}$ for all $m \geq 2,  r \in \{1, m-1\}$ where $p$ is an odd prime, and present a construction.  Our approach uses PDSs formed as unions of cyclotomic classes.
\end{abstract}

\section{Background and motivation}

A $k$-subset $D$ of a finite group $G$ (written additively with identity $0_G$) of order $v$ is a $(v,k,\lambda,\mu)$ {\em partial difference set} (PDS) if the multiset of differences $\Delta := \{ d_1 - d_2 : d_1, d_2 \in D \}$ contains every element of $D-\{0_G\}$ precisely $\lambda$ times and every element of $G-D-\{ 0_G \}$ precisely $\mu$ times. A PDS is equivalent to a strongly regular graph with a regular automorphism group (a sharply transitive automorphism group). A PDS can be thought of as a symmetric association scheme with two classes~\cite{Bos}, and there are connections to projective two-weight codes~\cite{CalKan}. Many infinite families have been constructed, both in abelian and nonabelian groups (see \cite{Ma_1994} for a good survey).  One general family of PDSs is the so-called Negative Latin Square type PDSs, with parameters $(n^2, r(n+1), -n+r+3r^2, r^2+r)$. We note that the order of the group must be a square in order for the PDS to be Negative Latin Square type.

The focus of this paper will be on a family originally due to Denniston~\cite{Den}; the parameters are $(2^{3m}, (2^{m+r} - 2^m + 2^r)(2^m-1), 2^m-2^r+(2^{m+r}-2^m+2^r)(2^r-2), (2^{m+r}-2^m+2^r)(2^r-1))$ for $m \geq 2, 1 \leq r < m$, and Denniston provided a construction in the elementary abelian group $\Z_2^{3m}$ for all of these possibilities. Denniston's examples correspond to maximal arcs in Desarguesian projective planes of even order. Davis and Xiang~\cite{DavXia} used Galois Rings to show that $\Z_4^m \times \Z_2^m$ contains a PDS for the $r \in \{1, m-1\}$ cases for all $m \geq 2$. Brady~\cite{Bra} used a computer search to find examples in the $m=2, r=1$ case for $51$ groups of order $64$, and Smith~\cite{Smi} extended this to an exhaustive search that found examples in $73$ out of $267$ nonisomorphic groups of order $64$. We note that PDSs with Denniston parameters will never be of Negative Latin Square type when $m$ is odd; when $m$ is even there will be one intermediate value of $r$ (namely $r=m/2$) corresponding to a Negative Latin Square type PDS.

We examine the question of whether, when $p$ is an odd prime, PDSs exist with Denniston parameters, i.e. whether there exist PDSs with parameters $(p^{3m}, (p^{m+r} - p^m + p^r)(p^m-1), p^m-p^r+(p^{m+r}-p^m+p^r)(p^r-2), (p^{m+r}-p^m+p^r)(p^r-1))$. We answer this question in the affirmative for the $m \geq 2; r \in \{1, m-1\}$ cases in the group $\Z_p^{3m}$. This is a rather surprising result since (as noted above) Denniston's original examples correspond to maximal arcs in Desarguesian projective planes of even order, but maximal arcs do not exist in Desarguesian projective planes of odd order \cite{BalBloMaz}.  We construct a partial difference set in $GF(p^m) \times GF(p^{2m})$ using cyclotomy; more details on cyclotomy in direct products of finite fields may be found for example in \cite{FerKwaMar}.

The paper is organized as follows. Section~\ref{quadraticform} introduces quadratic forms and uses these to construct a PDS in $GF(p^{2m})$.  Section~\ref{cyclotomy} reframes the main result of Section~\ref{quadraticform} in terms of cyclotomic classes. Section~\ref{charactertheory} defines characters of the group $\Z_p^{3m}$ and demonstrates how they can be used to prove that a subset of a finite abelian group is a PDS. Section~\ref{mainconstruction} contains the construction of the Denniston PDSs in $\Z_p^{3m}$ for all $m \geq 2$.  Finally, Section~\ref{futuredirections} contains questions to be considered in the future.

\section{Quadratic forms}
\label{quadraticform}

For the rest of the paper let $p$ be an odd prime, and consider $GF(p^{2m})$ as a degree-$2$ field extension of $GF(p^m)$, which is in turn a degree-$m$ extension of the prime field $GF(p)$.  Each of $GF(p^{2m})$ and $GF(p^m)$ can be viewed as vector spaces over $GF(p)$ (of dimension $2m$ and $m$ respectively).  

As an intermediate step to our main result, we use quadratic forms to form a PDS  in $GF(p^{2m})$, which can then be viewed as a union of cyclotomic classes with indices specified by a Singer difference set.  We note that a cyclotomic ``lifting construction'' for forming PDSs via subdifference sets of Singer difference sets appears in a strongly-regular graph setting in \cite{MomXia1} and \cite{MomXia} and could be applied in our case to establish this intermediate step.  However, to guarantee that our main construction is clear to the reader, we present a self-contained and direct proof of this intermediate PDS construction.

We have the following definition of a quadratic form (see Ma~\cite{Ma_1994}).

\begin{defn}
Let $V$ be an $n$-dimensional vector space over a field $F$. A function $Q: V \rightarrow F$ is a quadratic form if it satisfies:
\begin{itemize}
\item[(i)] For $ a \in F$, $Q(ax)=a^2 Q(x)$.
\item[(ii)] The function $R: V \times V \rightarrow F, R(u,v):=Q(u+v)-Q(u)-Q(v)$ is bilinear, i.e. linear in each argument.
\end{itemize}
Further, $Q$ is nondegenerate if $R$ is nondegenerate; i.e.
$R(u,v)=0$ for all $v \in V$ implies $u=0$.
\end{defn}

We present the following quadratic form which we will use in this paper. 

\begin{prop}
Let $Q: GF(p^{2m}) \rightarrow GF(p^m)$ be given by $Q(x)=x^{p^m+1}$.  Then $Q(x)$ is a nondegenerate quadratic form.
\end{prop}
\begin{proof}
We show that the above requirements are satisfied for $V=GF(p^{2m})$, $F = GF(p^m)$ and $Q(x)=x^{p^m+1}$.
First, let $a \in GF(p^m)$; then $a^{p^m}=a$.  So 
$$Q(ax)=(ax)^{p^m+1}=a^{p^m+1}x^{p^m+1}=a^2 x^{p^m+1} = a^2 Q(x).$$
We define the Frobenius automorphism $\phi_p: GF(p^{2m}) \rightarrow GF(p^{2m})$ by $\phi_p(x)=x^p$, and the fact that this is an automorphism~\cite{LidNie} implies that $(a+b)^{p^t} = a^{p^t} + b^{p^t}$ for $a,b \in GF(p^{2m}), t \geq 1$. This leads to the following computation.

\begin{eqnarray*}
R(u,v)=Q(u+v)-Q(u) - Q(v) & = & (u+v)^{p^m+1} - u^{p^m+1} - v^{p^m+1} \\
& = & (u+v)^{p^m}(u+v) - u^{p^m+1} - v^{p^m+1} \\
& = & (u^{p^m} + v^{p^m})(u+v) - u^{p^m+1} - v^{p^m+1} \\
& = & u^{p^m}v + uv^{p^m} \\
\end{eqnarray*}

Clearly the expression for $R$ is symmetric in $u$ and $v$.  We have

\begin{eqnarray*}
R(u_1+u_2,v) & = & (u_1+u_2)^{p^m}v+ (u_1+u_2)v^{p^m} \\
& = & (u_1^{p^m}+u_2^{p^m})v+(u_1+u_2)v^{p^m} \\
& = & u_1^{p^m}v+u_1 v^{p^m} + u_2^{p^m}v+u_2 v^{p^m} \\
& = & R(u_1,v)+R(u_2,v) \\
\end{eqnarray*}

For $a \in GF(p^m)$, we have 

\begin{eqnarray*}
R(au,v) & = & (au)^{p^m}v+(au)v^{p^m} \\
& = & a^{p^m}u^{p^m}v+auv^{p^m} \\
& = & au^{p^m}v+auv^{p^m} \\
& = & a R(u,v).\\
\end{eqnarray*}

So $Q$ satisfies both properties and hence is a quadratic form.
Finally, to prove it is nondegenerate: suppose $R(u,v)=0$ for all $v \in GF(p^{2m})$.  Then, for all $v$,
$u^{p^m}v+u v^{p^m}=0$.  Taking $v=u$ we see that $2u^{p^m+1}=0$, ie $u^{p^m+1}=0$ and hence $u=0$.
\end{proof}

Next, let $Tr_{m/1}: GF(q^m) \rightarrow GF(q)$ be the trace map, given by $Tr_{m/1}(x)=x+x^q+x^{q^2} + \cdots + x^{q^{m-1}}$.  We have the following result of Brouwer \cite{Bro}:

\begin{lemma}\label{BrouwerLemma}[Brouwer]
Let $F=GF(q^m)$ and $F_0=GF(q)$ ($q$ a prime power and $m>1$).  Let $V$ be a vector space of dimension $d$ over $F$ and write $V_0$ for the same vector space over $F_0$ (we will take $V=GF(q^{dm})$).
\begin{itemize}
\item[(i)] If $Q:V \rightarrow F$ is a quadratic form on $V$, then $Q_0=Tr \circ Q$ is a quadratic form on $V_0$.
\item[(ii)] $Q_0$ is nondegenerate if and only if $Q$ is nondegenerate and either $q$ is odd or $d$ is even.
\end{itemize}
\end{lemma}

\noindent Here we will take $F=GF(p^m)$, $F_0=GF(p)$ and $V=GF(p^{2m})$, so $d=2$.  
Let 
$$Q_0=Tr_{m/1}(Q(x))=Tr_{m/1}(x^{p^m+1})=x^{p^m+1}+x^{p^{m+1}+p}+x^{p^{m+2}+p^2}+ \cdots + x^{p^{2m-1} + p^{m-1}}.$$
Part (i) of Lemma~\ref{BrouwerLemma} implies that $Q_0$ is a quadratic form on $V_0=GF(p)$, and part (ii) tells us that $Q_0$ is nondegenerate (as $Q$ is nondegenerate and $q$ is odd).

Next, we state the following useful result (Theorem 2.6 of \cite{Ma_1994}):

\begin{thm}\label{MaTheorem}
Let $Q_0: \mathbb{F}_q^{2m} \rightarrow \mathbb{F}_q$ be a nondegenerate quadratic form.  Then
$$X=\{x \in \mathbb{F}_q^{2m} \setminus \{0\}: Q_0(x)=0\}$$ 
is a regular $(q^{2m}, q^{2m-1}+ \epsilon q^{m-1}(q-1)-1, q^{2m-2}+\epsilon q^{m-1}(q-1)-2, q^{2m-2}+\epsilon q^{m-1})$-PDS in the additive group of $\mathbb{F}_q^{2m}$, where $\epsilon= \pm 1$ and depends on the choice of $Q_0$.
\end{thm}

Applying Theorem \ref{MaTheorem} to our quadratic form $Q_0$ yields the following:
\begin{thm}\label{X_PDS}
Let $Q_0: \mathbb{F}_p^{2m} \rightarrow \mathbb{F}_p$, $Q_0(x)=Tr_{m/1}(x^{p^m+1})$.  Then
$$X=\{x \in GF(p^{2m}) \setminus \{0\}: Tr_{m/1}(x^{p^m+1})=0\}.$$ 
is a regular $(p^{2m}, p^{2m-1}-p^{m-1}(p-1)-1, p^{2m-2}-p^{m-1}(p-1)-2, p^{2m-2}-p^{m-1})$-PDS in the additive group of $\mathbb{F}_p^{2m}$.
\end{thm}
\begin{proof}
We apply Theorem \ref{MaTheorem}.  Using our choice of quadratic form $Q_0: GF(p^{2m}) \rightarrow GF(p)$, $Q_0(x)=Tr_{m/1}(x^{p^m+1})$, our set $X$ is defined by:
$$X=\{x \in GF(p^{2m})^*: Q_0(x)=0\}=\{x \in GF(p^{2m})^*: Tr_{m/1}(x^{p^m+1})=0\}.$$ 
In our setting, $\epsilon$ as defined in Theorem \ref{MaTheorem} takes the value $-1$.  
\end{proof}

We observe that $X$ is of Negative Latin Square type (with $n=p^m$ and $r=p^{m-1}-1$).

\begin{remark}
As noted at the start of this section, this PDS result could also be obtained as a special case of the ``lifting construction'' of \cite{MomXia1, MomXia}.  This would be achieved by an application of Theorem 3.1 of \cite{MomXia1}; this result is applicable since (in the notation of \cite{MomXia1}) we have a cyclotomic strongly regular graph $\mathrm{Cay}(\mathbb{F}_{p^m}, C_0)$ corresponding to the subfield case; specifically $C_0 \cong \mathbb{F}_p^*$.
\end{remark}

\section{Cyclotomy}
\label{cyclotomy}

In this section, we introduce the notion of cyclotomy, and show how Theorem \ref{X_PDS} may be expressed in terms of cyclotomic classes.  This is useful as it will provide a natural way to express our main construction for the PDS with Denniston parameters.  

\begin{defn}\label{cycldef}
Let $GF(q)$ be a finite field of order $q =p^n= ef + 1$ (where $p$ is prime and $e>1$). Let $\alpha$ be a primitive element of GF$(q)$. The \emph{cyclotomic classes} of order $e$  are defined: 
\begin{center}
$C_i^{e,n} = \alpha^i\langle\alpha^e\rangle$   
\end{center}
for $0 \leq i \leq e-1$. Each cyclotomic class is of size $f$.
\end{defn}

The study of cyclotomic classes originates with Gauss, and they have been used to construct difference sets and related objects since the work of Paley \cite{Pal}. There is particular interest in the question of when individual cyclotomic classes, or unions of such classes, form difference sets and PDSs; for a comprehensive recent survey, see \cite{MomWanXia}.  

In our setting, we take the field $GF(p^{2m})$ (where $p$ is an odd prime), and let $\alpha$ be a primitive element of this field.  Our classes have the form $C_i^{e,2m}$ where $e|p^{2m}-1$.  Note that $C_i^{e,2m} = \alpha^i C_0^{e,2m}$.

We now reframe the results of the previous section and show that the preimages under $Q$ of the powers of the generator of $GF(p^m)$ correspond to certain cyclotomic classes.  This yields a cyclotomic intepretation of the PDS $X$ from Theorem \ref{X_PDS}.

\begin{thm}\label{cycPDS}
Let $GF(p^{2m})^*=\langle \alpha \rangle$ and $GF(p^m)^*=\langle \omega \rangle$, where $\omega=\alpha^{p^m+1}$.  As in Section \ref{quadraticform}, let $Q: GF(p^{2m}) \rightarrow GF(p^m)$ be given by $Q(x)=x^{p^m+1}$, let $Q_0=Tr_{m/1}(Q(x))$ and let $X=\{x \in \mathbb{F}_q^{2m} \setminus \{0\}: Q_0(x)=0\}$.
\begin{itemize}
\item[(i)] The cyclotomic class $C_{i}^{p^m-1,2m}$ is the preimage under $Q$ of $\omega^i$ in $GF(p^m)^*$, where $0 \leq i \leq p^m-2$.
\item[(ii)] Let $I=\{i: Tr_{m,1}(\omega^i)=0, \quad 0 \leq i \leq p^m-2\}$.  Then 
$$X=\bigcup_{i \in I} C_{i}^{p^m-1,2m}.$$
\end{itemize}
\end{thm}
\begin{proof}
For (i), note $C_{i}^{p^m-1,2m}=\{ \alpha^{(p^m-1)t+i}: 0 \leq t \leq (p-1)(p^m+1)-1 \}$ ($0 \leq i \leq p^m-2$) and $1=\alpha^{p^{2m}-1}=\alpha^{(p^m-1)(p^m+1)}$. Then 
$$Q(\alpha^{(p^m-1)t+i})=(\alpha^{(p^m-1)t+i})^{p^m+1}=\alpha^{(p^m-1)(p^m+1)t} \cdot \alpha^{(p^m+1)i}=1 \cdot (\alpha^{p^m+1})^i=\omega^i.$$
Part (ii) is immediate from the definition of $X$ combined with part (i): the set $X$ consists of precisely those $C_{i}^{p^m-1,2m}$ such that $\omega^i$ are the trace-$0$ elements of $GF(p^m)$.
\end{proof}

Observe that $|I|=p^{m-1}-1$: the trace map $Tr_{m/1}$ is a linear map from $GF(p^m)$ with a kernel of size $p^{m-1}$ (the image has size $p$ and the map is onto), and the zero element is excluded from $I$.

Next, we will demonstrate a connection to Singer difference sets, and obtain a simpler expression for $X$.  Consider $GF(p^m)$ as a dimension-$m$ vector space over $GF(p)$ with basis $\{1,\omega,\omega^2 \ldots \omega^{m-1} \}$.  We will later be working in projective space, i.e. we will take the non-zero $m$-tuples and identify two $m$-tuples if they differ by a scalar factor.

For $GF(p^m)$, the zero-trace condition defines a hyperplane in the vector space.  A well-known result of Singer (Construction 18.28, \cite{ColDin}) tells us that, modulo the scalars, its elements form a multiplicative $(\frac{p^m-1}{p-1}, \frac{p^{m-1}-1}{p-1}, \frac{p^{m-2}-1}{p-1})$-difference set in the cyclic group of order $\frac{p^m-1}{p-1}$ induced by $\omega$ on the equivalence classes.  Let $e=\frac{p^m-1}{p-1}$.     Taking the set of powers $i$ such that $\omega^i$ is in this difference set, will yield a corresponding difference set in the additive group $(\mathbb{Z}_{e},+)$.   Observe that $GF(p) \setminus \{0\}$ is in fact the cyclotomic class of order $e$ in $GF(p^m)$, so we have factored by $\langle w^e \rangle=C_0^{e,m}$.

Let $T$ be the set of nonzero trace-$0$ elements of $GF(p^m)$.  
Since $|T|=p^{m-1}-1$, there are $\frac{p^{m-1}-1}{p-1}$ elements $S=\{ \omega^i: Tr_{m/1}(\omega^i)=0, 0 \leq i \leq e-1 \}$.  Taking the corresponding set of powers yields the set $D=\{i: \omega^i \in S,  0 \leq i \leq e-1 \}$,
an $(e, \frac{p^{m-1}-1}{p-1}, \frac{p^{m-2}-1}{p-1})$ Singer difference set in $(\mathbb{Z}_e,+)$. We record the following, which will be useful in what follows:

\begin{lemma}\label{traceT} The set of nonzero elements $T$ of $GF(p^m)$ with trace $0$ is
$$T=\cup_{j \in D} \{ \omega^j, \omega^{e+j}, \ldots, \omega^{(p-2)e+j}\}= \cup_{j \in D} w^j \langle w^e \rangle.$$
The set $I$ of indices of $T$ is
$$I=\cup_{j \in D} \{j,e+j,\ldots,(p-2)e+ j\}.$$
\end{lemma}

Finally, consider our set $X=\bigcup_{i \in I} C_{i}^{p^m-1,2m}.$  We can rewrite this as
$$X=\bigcup_{j \in D} C_{j}^{p^m-1,2m} \cup  C_{e+j}^{p^m-1,2m} \cup \cdots \cup  C_{(p-2)e+j}^{p^m-1,2m}$$
Observe that since $e=\frac{p^m-1}{p-1}$, $e|p^m-1$ and so 
$$C_{0}^{p^m-1,2m} \cup  C_{e}^{p^m-1,2m} \cup \cdots \cup  C_{(p-2)e}^{p^m-1, 2m}=C_{0}^{e,2m}$$ and more generally $C_{j}^{p^m-1,2m} \cup  C_{e+j}^{p^m-1,2m} \cup \cdots \cup  C_{(p-2)e+j}^{p^m-1,2m}=C_{j}^{e,2m}$ for $0 \leq j \leq p-2$.  So the $p-1$ cyclotomic classes of order $p^m-1$ corresponding to a fixed $j \in D$ partition a single cyclotomic class of order $e=\frac{p^m-1}{p-1}=p^{m-1}+p^{m-2}+ \cdots +p+1$.  Summarising, we have the following result:

\begin{thm}\label{PDSthm}
Let $p$ be an odd prime and $m>1$.  Let $D$ be the $(\frac{p^m-1}{p-1}, \frac{p^{m-1}-1}{p-1}, \frac{p^{m-2}-1}{p-1})$ Singer difference set in $(\mathbb{Z}_{\frac{p^m-1}{p-1}},+)$ defined by the trace-0 elements of $GF(p^m)$.  Denote by $C_i^{e,2m}$ the $i$-th cyclotomic class of order $e$ in $GF(p^{2m})$.
Then
\begin{itemize}
\item[(i)] $$X=\bigcup_{j \in D} C_{j}^{\frac{p^m-1}{p-1}, 2m}$$
 is a regular $(p^{2m}, p^{2m-1}-p^{m-1}(p-1)-1, p^{2m-2}-p^{m-1}(p-1)-2, p^{2m-2}-p^{m-1})$-PDS in $GF(p^{2m})$.
\item[(ii)] $X$ is precisely the PDS $\{x \in \mathbb{F}_q^{2m} \setminus \{0\}: Q_0(x)=0\}$ of Theorem \ref{X_PDS}.
\end{itemize}
\end{thm}

We end this section with the following technical corollary which will be used later:
\begin{cor}\label{X_k}
For $0 \leq k \leq \frac{p^m-1}{p-1}-1$, let $(-k)+D=\{(-k)+i: i \in D\}$, and define $$X_k=\bigcup_{j \in (-k)+D} C_{j}^{\frac{p^m-1}{p-1}, 2m}.$$
Then $X_k$ is a regular $(p^{2m}, p^{2m-1}-p^{m-1}(p-1)-1, p^{2m-2}-p^{m-1}(p-1)-2, p^{2m-2}-p^{m-1})$-PDS in $GF(p^{2m})$.
\end{cor}
\begin{proof}
Immediate since $X=\bigcup_{j \in D} \alpha^j C_0^{\frac{p^m-1}{p-1}, 2m}$ and 
$$X_k=\bigcup_{j \in (-k)+D} \alpha^j C_0^{\frac{p^m-1}{p-1}, 2m}=\bigcup_{i \in D} \alpha^{(-k)+i} C_0^{\frac{p^m-1}{p-1}, 2m}=\alpha^{-k}\bigcup_{i \in D} \alpha^i C_0^{\frac{p^m-1}{p-1}, 2m}=\alpha^{-k}X.$$
Since $X$ is a PDS comprising a union of cyclotomic classes of $GF(p^{2m})$, then $\alpha^{-k}X$ is a PDS with the same parameters.
\end{proof}

\section{Character theory}
\label{charactertheory}

Let $G$ be a finite abelian group. A {\em character} on $G$, often denoted $\chi$, is a homomorphism from $G$ to the group of complex numbers under multiplication. The fact that $\chi$ is a homomorphism implies that for all $g \in G, \chi(g)$ must be an $n^{th}$ root of unity for $n$ an integer satisfying $ng = 0_G$.  The principal character $\chi_0$ is the character that maps all elements of $G$ to $1 \in \C$; all other characters are called nonprincipal.

Consider the field $F = GF(p^m)$, and let $\omega$ be a primitive element (i.e. a generator of the multiplicative group of $F$). All nonprincipal (additive) characters of $F$ have the form $\chi_k(x) = (e^{2 \pi i/p})^{Tr_{m/1}(\omega^k x)}$ for $0 \leq k \leq p^m-2$. Similarly, if $K = GF(p^{2m})$ has primitive element $\alpha$, then all nonprincipal (additive) characters of $K$ will have the form $\psi_{\ell}(x) = (e^{2 \pi i/p}
)^{Tr(\alpha^{\ell}x)}$ for $0 \leq \ell \leq p^{2m}-2$.

Character theory has been one of the primary tools used to investigate PDSs because of the following theorem (see~\cite{Ma_1994}).

\begin{thm}
\label{characterTHPDSs}
Let $G$ be a finite abelian group of order $v$. Suppose $k, \lambda$ and $\mu$ are positive integers satisfying $k^2 = k + \lambda k + \mu (v-k-1)$.   Let $D$ be a $k$-element subset of $G$ such that $D=-D$, and $D$ does not contain the identity.  Then $D$ is a $(v,k,\lambda,\mu)$ PDS in $G$ if and only if  for all nonprincipal characters $\chi$ on $G$ we have $\chi(D) \in \{ 1/2((\lambda-\mu) \pm \sqrt{(\mu-\lambda)^2 + 4(k-\mu)} ) \}$.
\end{thm}

When we combine Theorems~\ref{MaTheorem} and~\ref{X_k}, we see that the sets $X_k$ defined in Corollary \ref{X_k} will have nonprincipal character values of either $p^{m-1}-1$ or $-[(p-1)p^{m-1}+1]$. We will use this fact in the next section to prove our main construction.

When working with PDSs in abelian groups, Delsarte~\cite{Del} introduced the {\em dual PDS}. The set of characters $G^*:= \{ \chi:G \rightarrow \C|\chi \mbox{ is a homomorphism} \}$ is a group under multiplication, and a simple computation demonstrates that $G^*$ is isomorphic to $G$. If $D$ is a $(v,k,\lambda, \mu)$ PDS in $G$, we define $D^*:= \{ \chi \in G^* : \chi(D) = 1/2((\lambda-\mu) + \sqrt{(\mu-\lambda)^2 + 4(k-\mu)} ) \}$. 

\begin{thm}
\label{dualPDS}
If $D$ is a $(v,k,\lambda, \mu)$ PDS in the abelian group $G$ and $D^* = \{ \chi \in G^* : \chi(D) = 1/2((\lambda-\mu) + \sqrt{(\mu-\lambda)^2 + 4(k-\mu)} ) \} \subset G^*$, then $D^*$ is a $(v, k^*, \lambda^*, \mu^*)$ PDS in $G^*$ for 

\begin{eqnarray*} \beta & = & \lambda-\mu \\
\gamma & = & k-\lambda \\
\Delta & = & \beta^2+4 \gamma \\
k^* & = & [(\sqrt{\Delta}-\beta)(v-1)-2k]/(2 \sqrt{\Delta}) \\
\lambda^* & = & (v-2k+\beta-\sqrt{\Delta})/\sqrt{\Delta} + [4k^*-v^2/\sqrt{\Delta} + (v-2k+\beta-\sqrt{\Delta}/\sqrt{\Delta})^2]/4 \\
\mu^* & = & [4k^* - v^2/\Delta +(v-2k+\beta-\sqrt{\Delta})^2/(\Delta)]/4. \\
\end{eqnarray*}
\end{thm}

In our situation, we want to construct a PDS in $\Z_p^{3m}$ with the following parameters 

$$(p^{3m}, (p^{m+1} - p^m + p)(p^m-1), p^m-p+(p^{m+1}-p^m+p)(p-2), (p^{m+1}-p^m+p)(p-1)).$$

\noindent This is the Denniston $r=1$ case. If we are able to construct this PDS, then Theorem~\ref{dualPDS} implies that the dual PDS will be the Denniston $r=m-1$ case, namely 

$$(p^{3m}, (p^{2m-1}-p^m+p^{m-1})(p^m-1), p^m-p^{m-1}+(p^{2m-1}-p^m+p^{m-1})(p^{m-1}-2),$$
$$(p^{2m-1}-p^m + p^{m-1})(p^{m-1}-1)).$$ 

For the $r=1$ Denniston case, the character values from Theorem~\ref{characterTHPDSs} are $p^m-p$ and $-((p-1)(p^m+1)+1)$. In the next section we define a set that will have one of these two values for all nonprincipal characters.

\section{Main construction}
\label{mainconstruction}

We now turn our attention to the general case, by which we mean the following: our group will be $GF(p^m)^+ \times GF(p^{2m})^+ \cong \Z_p^{3m}$, and we will use $\omega$ as a primitive element of $GF(p^m)$ and $\alpha$ as a primitive element of $GF(p^{2m})$.   We define 

$${\mathcal D}:= \displaystyle \bigcup_{i=0}^{\frac{p^m-1}{p-1}-1} (\omega^i \langle \omega^{\frac{p^m-1}{p-1}}  \rangle) \times (C_{i}^{\frac{p^m-1}{p-1},2m} \cup \{ 0_{GF(p^{2m})} \}).$$

The next theorem is the main result in the paper, and it demonstrates that Denniston PDSs exist in the $p$ odd case for $r=1$.

\begin{thm}
\label{mainresultsmallcase}
The set ${\mathcal D}$ is a $(p^{3m}, (p^m-1) \cdot ((p-1)(p^m+1)+1), p^m-p+(p^{m+1}-p^m+p)(p-2), (p^{m+1}-p^m+p)(p-1))$ PDS in $GF(p^m)^+ \times GF(p^{2m})^+$.
\end{thm}

\begin{proof}
In this proof we will write $C_{i}^{(p^m-1)/(p-1),2m}$ as $C_{i}^{(p^m-1)/(p-1)}$ for clarity.\\
Every character $\Phi$ of our group $GF(p^m)^+ \times GF(p^{2m})^+$ can be written as $\Phi = \chi_k \times \psi_{\ell}$ for $\chi_k$ a character of $GF(p^m)^+$ and $\psi_{\ell}$ a character of $GF(p^{2m})^+$. 
We break the proof into three cases based on the characters of the group.
\begin{description}
\item[Case 1: $\chi_k$ principal, $\psi_{\ell}$ nonprincipal] In this case, $\chi_k(\omega^i \langle \omega^{\frac{p^m-1}{p-1}} \rangle) = p-1$ for $0 \leq i \leq \frac{p^m-1}{p-1}-1$. Thus, 

\begin{eqnarray*}
(\chi_k \times \psi_{\ell})({\mathcal D}) & = & (p-1) \displaystyle \sum_{i=0}^{\frac{p^m-1}{p-1}-1} \psi_{\ell}(C_i^{\frac{p^m-1}{p-1}-1} \cup \{ 0_{GF(p^{2m})} \}) \\
& = & (p-1) \psi_{\ell}(\sum_{i=0}^{\frac{p^m-1}{p-1}-1} C_i^{\frac{p^m-1}{p-1}-1}) + (p-1) \cdot \frac{p^m-1}{p-1} \\
& = & (p-1)(-1) + p^m-1 = p^m-p. \\
\end{eqnarray*}

\item[Case 2: $\chi_k$ nonprincipal, $\psi_{\ell}$ principal] In this case, $\psi_{\ell}(C_i^{\frac{p^m-1}{p-1}} \cup \{ 0_{GF(p^{2m})} \}) = (p-1)(p^m+1)+1$ for $0 \leq i \leq \frac{p^m-1}{p-1}-1$. Thus,

\begin{eqnarray*}
(\chi_k \times \psi_{\ell})({\mathcal D}) & = &  \displaystyle ((p-1)(p^m+1)+1) \sum_{i=0}^{\frac{p^m-1}{p-1}-1} \chi_k( \omega^i \langle \omega^{\frac{p^m-1}{p-1}} \rangle) \\
& = & -((p-1)(p^m+1)+1) \\
\end{eqnarray*}

\item[Case 3: $\chi_k$ nonprincipal, $\psi_{\ell}$ nonprincipal] 

The kernel of $\chi_k$ is 
$$Ker(\chi_k) = \{ x \in GF(p^m) : Tr(\omega^kx) =  0_{GF(p^{m})} \},$$ 
and the size of the kernel is $p^{m-1}$. Now, $x \in Ker(\chi_k)$ precisely if $x=0_{GF(p^{m})}$ or $\omega^kx \in T$. So, by Lemma 3.3, we have $Ker(\chi_k) = \omega^{-k}T \cup \{ 0_{GF(p^{m})} \} = \cup_{i \in D} \omega^{-k+i} \langle \omega^{\frac{p^m-1}{p-1}} \rangle \cup \{ 0_{GF(p^m)} \}$. 

Let $\overline{k} \equiv k \mod \frac{p^m-1}{p-1}$;  then $Ker(\chi_k) = \cup_{j \in (-\overline{k})+D} \omega^j \langle \omega^{\frac{p^m-1}{p-1}} \rangle \cup \{ 0_{GF(p^m)} \}$ where $(-\overline{k})+D = \{ (-\overline{k})+i : i \in D \}$. A consequence of this is the following:

\begin{equation*}
 \chi_k(\omega^{j} \langle \omega^{\frac{p^m-1}{p-1}} \rangle) = \begin{cases}

p-1 & j \in (-\overline{k})+D \\

          -1 & \mbox{otherwise}

       \end{cases}
\end{equation*}

Similarly, the set $\cup_{j \in (-\overline{k})+D} C_j^{\frac{p^m-1}{p-1}}$ (in other words, the union of classes in $GF(p^{2m})$ corresponding to the indices of the elements of $Ker(\chi_k)$) is precisely $X_{\overline{k}}$ as defined in Corollary 3.5, and combining that with Theorem 2.4 we get

\begin{equation*}
 \psi_{\ell}(X_{\overline{k}}) = \begin{cases}

-[(p-1)p^{m-1} + 1] & \\

          p^{m-1}-1 &

       \end{cases}
\end{equation*}

Using the fact that $\psi_{\ell}(GF(p^{2m}))=0$, we then have 
\begin{equation*}
\psi_{\ell}(GF(p^{2m}) - X_{\overline{k}}-\{ 0_{GF(p^{2m})} \}) = \begin{cases}

(p-1)p^{m-1}  & \\

          -p^{m-1} &

       \end{cases}
\end{equation*}

Combining these results, we obtain


$$(\chi_k \times \psi_{\ell})({\mathcal D})  =  \sum_{i=0}^{\frac{p^m-1}{p-1}-1} \chi_k(\omega^i \langle \omega^{\frac{p^m-1}{p-1}} \rangle) \psi_{\ell}(C_i^{(p^m-1)/(p-1)} \cup \{ 0_{GF(p^{2m})} \}) $$

$$ = (p-1) (\psi_{\ell}(X_{\overline{k}})+\frac{p^{m-1}-1}{p-1}) + (-1) \left(\psi_{\ell}(GF(p^{2m}) - X_{\overline{k}}-\{ 0_{GF(p^{2m})} \}) + (\frac{p^m-1}{p-1} - \frac{p^{m-1}-1}{p-1})\right)$$

$$ = \begin{cases} (p-1) (-[(p-1)p^{m-1} + 1]+\frac{p^{m-1}-1}{p-1}) + (-1) [(p-1)p^{m-1}+(\frac{p^m-1}{p-1} - \frac{p^{m-1}-1}{p-1})] & \\

(p-1) (p^{m-1}-1+\frac{p^{m-1}-1}{p-1}) + (-1) (-p^{m-1}+(\frac{p^m-1}{p-1}-\frac{p^{m-1}-1}{p-1})) & \end{cases} $$

$$ = \begin{cases} -p^m(p-1)-p & \\

p^m-p & \end{cases} $$
\end{description}

Since every nonprincipal character $\chi_k \times \psi_{\ell}$ has a sum over ${\mathcal D}$ of either $-p^m(p-1)-p$ or $p^m-p$, Theorem~\ref{characterTHPDSs} implies that ${\mathcal D}$ is a PDS as claimed. \end{proof}

We note that the sets of $(k,\ell)$ pairs which correspond to each of the two possible values of $ \psi_{\ell}(X_{\overline{k}})$ in the final case may be determined by the techniques of \cite{MomXia1}.

Finally, an application of Theorem \ref{dualPDS} yields:
\begin{thm}
The dual of ${\mathcal D}$ is a $(p^{3m}, (p^{2m-1}-p^m+p^{m-1})(p^m-1), p^m-p^{m-1}+(p^{2m-1}-p^m+p^{m-1})(p^{m-1}-2),(p^{2m-1}-p^m + p^{m-1})(p^{m-1}-1))$-PDS  in $GF(p^m)^+ \times GF(p^{2m})^+$.
\end{thm}

\section{Future directions}
\label{futuredirections}

In this paper we have expanded the groups containing PDSs with the Denniston parameters to include odd prime powers. This opens several possible areas for further exploration.

\begin{itemize}
\item In elementary abelian $p$-groups of order $p^{3m}$ ($p$ an odd prime), construct PDSs with Denniston parameter sets for $2 \leq r \leq m-2$.
\item Do there exist PDSs with Denniston parameters in non-elementary abelian $p$-groups when $p$ is an odd prime?  As indicated in the Introduction, constructions exist in the $p=2$ case.
\item If there are new PDSs, examine the graphs to see whether they are isomorphic to known graphs with those parameters.
\item Can we find Denniston-like PDSs in groups that are not prime powers?
\item What other conclusions can we draw about unions of cyclotomic classes?
\end{itemize}

Since our work was announced, we have become aware that an equivalent result has simultaneously been proved by de Winter in the context of projective two-weight sets, using different (geometric) techniques \cite{deW}.  It is noted in \cite{deW} that a related coding theory result in the context of two-weight codes is contained in a paper by Bierbrauer and Edel from 1997 \cite{BieEde}, whose significance has not been hitherto recognised by the community.

We further note that the first open question in our list above has very recently been resolved by Bao et al in \cite{BaoXiaZha}, using cyclotomic methods.

\section{Acknowledgements}
We thank Ken Smith for assistance with computation.  The research of Sophie Huczynska was supported by EPSRC grant EP/X021157/1.  Laura Johnson's research visit to the University of Richmond was funded by LMS Early Career Research Travel Grant ECR-2223-10. John Polhill is grateful for the support from the Gaines Chair at the University of Richmond.  We thank the anonymous referees for their helpful comments.


\begin{thebibliography}{10}

\bibitem{BalBloMaz}
S. Ball, A. Blokhuis and F. Mazzocca.
\newblock Maximal arcs in Desarguesian planes of odd order do not exist. 
\newblock{\em Combinatorica}, 17 (1997) 31--41.

\bibitem{BaoXiaZha} 
J. Bao, Q. Xiang and M. Zhao.
\newblock Partial Difference Sets with Denniston Parameters in
Elementary Abelian p-Groups.
\newblock{ArXiV preprint- 2407.15632}

\bibitem{BieEde}
J. Bierbrauer and Y. Edel.
\newblock A family of 2-weight codes related to BCH-codes.
\newblock{ \em J. Combin. Des.} 5 (1997) 391–396.

\bibitem{Bos}
R.C. Bose.
\newblock Strongly regular graphs, partial geometries and partially balanced designs.
\newblock {\em Pacific J. Math.} 13 (1963) 389--419.


\bibitem{Bra}
A. C. Brady.
\newblock Negative Latin square type partial difference sets in nonabelian
  groups of order 64.
\newblock {\em Finite Fields Appl.}, 81 (2022) Paper No. 102044.

\bibitem{Bro}
A. E. Brouwer.
\newblock Some new two-weight codes and strongly regular graphs.
\newblock {\em Discrete Applied Mathematics} 10 (1985) 111-114.


\bibitem{CalKan}
R Calderbank and W. Kantor.
\newblock The geometry of two-weight codes.
\newblock {\em Bull. London Math. Soc.} 18 (1986) 97-122.

\bibitem{ColDin}
C. J. Colbourn and J. H. Dinitz.
\newblock Handbook of Combinatorial Designs (2nd Edition).
\newblock Chapman and Hall/CRC, 2006.

\bibitem{DavXia}
James~A. Davis and Qing Xiang.
\newblock A family of partial difference sets with Denniston parameters in nonelementary abelian 2-groups.
\newblock {\em Eur. J. Comb.} 21 (2000) 981-988.

\bibitem{deW}
S. de Winter.
\newblock Projective Two-Weight Sets of Denniston Type.
\newblock{ ArXiV preprint 2311.00827}


\bibitem{Del}
P. Delsarte. 
\newblock An algebraic approach to the association schemes of coding theory. Philips Research Report.
\newblock 1973, Suppl. No. 10.
  
\bibitem{Den}
R.~Denniston.
\newblock Some maximal arcs in finite projective planes.
\newblock {\em J. Comb. Th.} 6 (1969) 317--319.

\bibitem{FerKwaMar}
G.A. Fernandez-Alcober, R. Kwashira and L.Martinez.
\newblock Cyclotomy over products of finite fields and combinatorial applications.
\newblock{\em Europ. J. Combin.} 31 (2010) 1520-1538. 

\bibitem{LidNie}
R. Lidl and H. Niederreiter.
\newblock Introduction to finite fields and their applications.
\newblock Cambridge University Press, 2002.

\bibitem{Ma_1994}
S.~L. Ma.
\newblock A survey of partial difference sets.
\newblock {\em Des. Codes Cryptogr.} 4 (1994) 221--261.

\bibitem{MomXia1} 
K. Momihara and Q. Xiang.
\newblock Lifting constructions of strongly regular Cayley graphs.
\newblock{\em Finite Fields Appl.} 26 (2014) 86–99.

\bibitem{MomXia} K. Momihara and Q. Xiang.
\newblock Strongly regular Cayley graphs from partitions of subdifference sets of the Singer difference sets.
\newblock {\em Finite Fields Appl.} 50 (2018) 222-250.

\bibitem{MomWanXia}
K. Momihara, Qi Wang and Qing Xiang.
\newblock Cyclotomy, difference sets, sequences with low correlation, strongly regular graphs and related geometric substructures.
\newblock {\em Combinatorics and Finite Fields: Difference Sets, Polynomials, Pseudorandomness and Applications} edited by Kai-Uwe Schmidt and Arne Winterhof, Berlin, Boston: De Gruyter, 2019, pp. 173-198. 

\bibitem{Pal}
R.E.A.C.Paley.
\newblock On orthogonal matrices.
\newblock{\em Stud. Appl. Math.} 12 (1933) 311-320.

\bibitem{Smi}
K.W. Smith
\newblock personal communication, 2023.

\end{thebibliography}
\end{document}